\documentclass[a4paper,12pt]{article}
\usepackage{lineno}
\usepackage{hyperref}
\usepackage{mathrsfs}
\usepackage{amssymb}
\usepackage{amsmath}
\usepackage{amsfonts}
\usepackage{amsthm}
\usepackage{amscd}
\usepackage{enumerate}
\usepackage[all]{xy}

\theoremstyle{definition}
\newtheorem{defn}{Definition}[section]
\newtheorem{lemma}[defn]{Lemma}
\newtheorem{theorem}[defn]{Theorem}
\newtheorem{proposition}[defn]{Proposition}

\newtheorem{remark}[defn]{Remark}

\begin{document}

\title{Explicit $K_2$ of certain quotient rings over $\mathbb{Z}[G]$ for $G$ an elementary abelian $p$-group}
\author{
\textbf{Yakun Zhang}\\
\small \emph{School of Mathematics, Nanjing Audit University,}\\
\small \emph{Nanjing 211815, China}\\
\small \emph{E-mail: zhangyakun@nau.edu.cn}\\
}
\date{}
\maketitle

\noindent \textbf{Abstract.} Let $G$ be an elementary abelian $p$-group. In this paper, we calculate the $K_2$-groups of some quotient rings $\mathbb{Z}[G]/I$ for certain ideals $I \subseteq \mathbb{Z}[G]$ of finite $p$-power index. These results are established through the explicit computation of Dennis-Stein symbols. As an application, we provide a definitive characterization of the relative group $SK_1(\mathbb{Z}[G], p^k\mathbb{Z}[G])$ for any odd prime $p$ and $k \ge 1$.

\vspace{10pt}



\section{Introduction}
Let $G$ be a finite abelian group of order $p^n$, and let $\Gamma$ be the maximal $\mathbb{Z}$-order of $\mathbb{Z}[G]$. By Proposition 2.2 in \cite{alperin1985sk}, $\Gamma$ is isomorphic to a direct product of rings of algebraic integers, and $p^n\Gamma \subseteq \mathbb{Z}[G]\subseteq \Gamma$. In \cite{ChenHong2014}, the explicit description of $\Gamma$ (see \cite[Lemma 3.3]{ChenHong2014}) and a lower bound of the order of $K_2(\mathbb{Z}[G]/p^n\Gamma)$ are given, where $G$ is an elementary abelian $p$-group or a cyclic $p$-group. However, the explicit structure of $K_2(\mathbb{Z}[G]/p^n\Gamma)$ remains an open problem.

Since $K_2(\mathbb{Z}[G]/p^n\Gamma)$ sits in the Mayer-Vietoris sequence:
\begin{equation}\label{mayer0}
K_2(\mathbb{Z}[G]){\rightarrow}K_2(\mathbb{Z}[G]/p^{n}\Gamma)\oplus K_2(\Gamma) {\rightarrow}K_2(\Gamma/p^n\Gamma){\rightarrow} SK_1(\mathbb{Z}[G]){\rightarrow} 1,
\end{equation}
one motivation for studying this group is that $K_2(\mathbb{Z}[G]/p^n\Gamma)$ can be used to give a lower bound of the order of $K_2(\mathbb{Z}[G])$, which has important applications in both topology and algebra. For any $k\geq n$, the relationship between $\mathbb{Z}[G]$ and $\mathbb{Z}[G]/p^k\Gamma$ can be described by the following Cartesian square of rings:
\begin{equation*}
\begin{CD}
\mathbb{Z}[G] @>>>  \Gamma  \\
@VVV @VVV  \\
\mathbb{Z}[G]/p^k\Gamma @>>>  \Gamma/p^k\Gamma
\end{CD}
\end{equation*}
By \cite[Theorem 1.3]{alperin1985sk}, the square induces a Mayer-Vietoris sequence which is a generalisation of \eqref{mayer0}:
\begin{equation} \label{new-sequence}
K_2(\mathbb{Z}[G]){\rightarrow}K_2(\mathbb{Z}[G]/p^{k}\Gamma)\oplus K_2(\Gamma) {\rightarrow}K_2(\Gamma/p^k\Gamma){\rightarrow} SK_1(\mathbb{Z}[G]){\rightarrow} 1.
\end{equation}
However, due to the complex structure of $\Gamma$, it is hard to determine the explicit structure of $K_2(\mathbb{Z}[G]/p^k\Gamma)$ directly. In \cite{Yakun2019}, $K_2(\mathbb{Z}[G])$ was investigated through studying the inverse limit of \eqref{new-sequence} to avoid this difficult problem. In this paper, we will focus on computing the $K_2$ of some quotient rings of $\mathbb{Z}[G]$.

Recall that the pro-$p$-group $K_2^{c}(\widehat{\mathbb{Z}}_p[G])$ is defined to be the inverse limit of the groups $K_2(\mathbb{Z}[G]/I)$ taken over all ideals $I\subseteq \mathbb{Z}[G]$ of finite $p$-power index (see \cite{oliver1987k}), where $\widehat{\mathbb{Z}}_p$ denotes the $p$-adic integers. By \cite[Lemma 3.8]{oliver1987k}, when $G$ is an elementary abelian $p$-group of rank $n$, we have the following isomorphism:
\begin{equation*}
K_2^{c}(\widehat{\mathbb{Z}}_p[G])\cong K_2^{c}(\widehat{\mathbb{Z}}_p)\oplus HC_1(\widehat{\mathbb{Z}}_p[G]),
\end{equation*}
where $HC_1(\widehat{\mathbb{Z}}_p[G])$ denotes the first cyclic homology of $\widehat{\mathbb{Z}}_p[G]$ (see \cite[p.~532]{oliver1987k}). By definition, we have
$$HC_1(\widehat{\mathbb{Z}}_p[G]) \cong (G\otimes \widehat{\mathbb{Z}}_p[G])/\langle g\otimes \lambda g: g\in G, \lambda\in \widehat{\mathbb{Z}}_p \rangle \cong \prod_{g \in G} G/\langle g\rangle,$$
thus the order of $K_2^{c}(\widehat{\mathbb{Z}}_p[G])$ is closely related to the number of cyclic subgroups of $G$:
\begin{equation}\label{K2c-order}
|K_2^{c}(\widehat{\mathbb{Z}}_p[G])|= |K_2^{c}(\widehat{\mathbb{Z}}_p)| \cdot \prod_{g\in G}|G/\langle g\rangle|.
\end{equation}
On the other hand, by definition,
$$K_2^{c}(\widehat{\mathbb{Z}}_p[G])\cong \lim\limits_{\overset{\longleftarrow}{k}} K_2(\mathbb{Z}[G]/(p^{k}))\cong \lim\limits_{\overset{\longleftarrow}{k\geq n}} K_2(\mathbb{Z}[G]/p^{k}\Gamma),$$
in which $K_2(\mathbb{Z}[G]/(p^k))$ was firstly investigated in \cite{alperin1987sk, alperin1985sk} to compute $SK_1(\mathbb{Z}[G])$ and is more accessible to computation. The explicit structures of $K_2(\mathbb{Z}[G]/(p^k))$ and $K_2(\mathbb{Z}[G]/p^{k}\Gamma)$ are actually dependent on $k$, and \eqref{K2c-order} gives the upper bound of the order of them. Since surjective homomorphism between two commutative semi-local rings induces a surjection on $K_2$ (see Lemma \ref{K2-surjection}), for a sufficiently large number $k$, the two $K_2$-groups are stable, i.e., isomorphic to their common inverse limit. Then three questions naturally arise:

\begin{enumerate}[(i)]
 \item How can one identify the minimal generators for $K_2(\mathbb{Z}[G]/p^k\Gamma)$ and $K_2(\mathbb{Z}[G]/p^k)$?
 \item Under what conditions do $K_2(\mathbb{Z}[G]/p^k\Gamma)$ and $K_2(\mathbb{Z}[G]/p^k)$ satisfy stability?
 \item Under what conditions are $K_2(\mathbb{Z}[G]/p^k\Gamma)$ and $K_2(\mathbb{Z}[G]/p^k)$ isomorphic to each other?
\end{enumerate}

For simplicity, we primarily consider the case $G$ an elementary abelian $p$-group of rank $n$. For $k=n=1$, the solutions are self-evident, as the two $K_2$-groups are inherently trivial (see \cite[p.~1579]{ChenHong2014}). To give a complete answer to the above questions, the article is organized as follows: we first give the explicit structure of $K_2(\mathbb{F}_p[G])$ in terms of Dennis-Stein symbols or Steinberg symbols (see Lemma \ref{generators}), then apply this result to give the explicit structure of $K_2(\mathbb{Z}[G]/(p^k))$ (see Theorem \ref{nontrivial2}) and $K_2(\mathbb{Z}[G]/p^k\Gamma)$ (see Theorem \ref{main-thm}). Our findings show that $K_2(\mathbb{Z}[G]/(p^k))$ exhibits stability for $k \geq 2$, and $K_2(\mathbb{Z}[G]/p^k\Gamma)$ maintains stability for $k \geq n+1$. Given that the two $K_2$-groups converge to the same inverse limit, and considering the surjective mapping from $K_2(\mathbb{Z}[G]/(p^k))$ to $K_2(\mathbb{Z}[G]/p^k\Gamma)$, it follows that they are isomorphic for $k \geq n+1$. As an application, for $p$ an odd prime, we compute $SK_1(\mathbb{Z}[G], p^k\mathbb{Z}[G])$ for all $k\geq 1$ (see Theorem \ref{relative-sk1}).
 
\noindent\textbf{Notations}

Let $R$ be a ring with identity. When $a, b \in R$ are commuting elements such that $1-ab \in R^*$, the Dennis-Stein symbol $\langle a,b \rangle$ is defined in $K_2(R)$. In this paper, we adopt the ``symmetric'' Dennis-Stein symbol as defined in \cite{stein1981maps}; thus, our $\langle a,b \rangle$ corresponds to $\langle -a, b \rangle$ in \cite{alperin1987sk, silvester1981introduction}.

The relative $K_2$-group $K_2(R,I)$ is generated by the Dennis-Stein symbols $\langle a,b \rangle$ where at least one of $a$ or $b$ lies in an ideal $I \subseteq R$. These symbols satisfy the following relations:
\begin{align*}
  (\text{DS1}) \quad &\langle a, b\rangle \langle a, c\rangle = \langle a, b+c-abc\rangle;\\
  (\text{DS2}) \quad &\langle a, b\rangle \langle b, a\rangle = 1; \\
  (\text{DS3}) \quad &\langle a, bc\rangle = \langle ab, c\rangle \langle ac, b\rangle.
\end{align*}
If $a \in R^*$, then $\langle a, b\rangle = \{a, 1-ab\}$.

Let $C_n$ denote a multiplicative cyclic group of order $n$. Unless otherwise specified, $G$ is an elementary abelian $p$-group of rank $n$, denoted by $G = \prod_{i=1}^{n} G_{i}$, where $G_i = \langle \sigma_i \mid \sigma_i^p=1 \rangle$. Let $\Gamma$ and $\Gamma_{i}$ denote the maximal $\mathbb{Z}$-order of $\mathbb{Z}[G]$ and $\mathbb{Z}[G_i]$, respectively. Let $x_i = \sigma_i - 1$. For an ideal $I \subseteq \mathbb{Z}[G]$ of finite $p$-power index, we define the following sets to identify the minimal generators of $K_2(\mathbb{Z}[G]/I)$ (with some abuse of notation in the subsequent computations):
\begin{align*}
T &= \{ \langle x_i, \prod_{j=1}^{n}x_j^{\lambda_j} \rangle \mid 1 \leq i \leq n, \, 0 \leq \lambda_j \leq p-1, \, \sum \lambda_j > 0 \}, \\
T_1 &= \bigcup_{i=1}^{n} \{ \langle x_i, \prod_{j=i}^{n}x_j^{\lambda_j} \rangle \mid 0 \leq \lambda_i \leq p-2, \, 0 \leq \lambda_{i+1}, \dots, \lambda_n \leq p-1, \, \sum \lambda_j > 0 \}, \\
T_2 &= \{ \langle x_i, x_i^{p-1} \rangle \mid 1 \leq i \leq n \}, \quad T_3 = \{ \langle x_i, \prod_{j=1}^{n} x_j^{p-1} \rangle \mid 1 \leq i \leq n \}.
\end{align*}
Finally, for any prime $p$, we define the indicator function
\[
\tau(p) = \begin{cases} 
1, & \text{if } p=2, \\
0, & \text{if } p \neq 2.
\end{cases}
\]

\section{The explicit structure of $K_2(\mathbb{Z}[G]/(p^k))$ for $k \geq 2$}

\begin{lemma}\cite[Proposition 1.1]{alperin1985sk} \label{K2-surjection}
Let $f: A \to B$ be a surjective homomorphism between two commutative semi-local rings. Then $f$ induces a surjection $K_2(A) \to K_2(B)$.
\end{lemma}

\begin{lemma} \cite[p.~255]{alperin1987sk} \label{scholium}
Let $\alpha_0, \dots, \alpha_l$ be elements of a ring such that $1 - \alpha_0 \cdots \alpha_l$ is invertible. Let $\hat{\alpha}_i = \alpha_0 \cdots \alpha_{i-1} \alpha_{i+1} \cdots \alpha_l$. Then
\begin{equation*}
    1 = \langle 1, \alpha_0 \cdots \alpha_l \rangle = \prod_{i=0}^l \langle \alpha_i, \hat{\alpha}_i \rangle.
\end{equation*}
\end{lemma}

The following result is a special case of the results established in \cite{gao2015explicit, zhang2019some}. Nevertheless, we provide an alternative line of argument here.

\begin{lemma}\label{generators}
$K_2(\mathbb{F}_p[G])$ is an elementary abelian $p$-group of rank $(n-1)(p^n-1)$. Furthermore, the set $T \setminus (T_1 \cup T_2)$ forms a basis for $K_2(\mathbb{F}_p[G])$.
\end{lemma}

\begin{proof}
According to Proposition 3 and Proposition 5 in \cite{dennis1976lower}, $K_2(\mathbb{F}_p[G])$ is an elementary abelian $p$-group, and the $p$-rank of $K_2(\mathbb{F}_p[G])$ is exactly $(n-1)(p^n-1)$. Set
$$S =\{\prod_{j=1}^{n}x_{j}^{\lambda_j} \mid 0 \leq \lambda_{j} \leq p-1, \sum_{j=1}^{n} \lambda_j > 0 \}, \quad
S_i =\{x_{i}^{\lambda_i} \mid 0 < \lambda_{i} \leq p-1\}.$$
Then by Lemma \ref{scholium}, for $s \in S_i,$ $\langle x_i, s \rangle$ is trivial. Note that $\langle -p,-p \rangle=0$ and $x_j^{p} = 0$. Moreover, Proposition 1.1 in \cite{alperin1987sk} also holds for $k=1$ with slight modifications. Thus, $K_2(\mathbb{F}_{p}[G])$ is generated by:
\begin{enumerate}[(i)]
  \item all Dennis-Stein symbols $\langle x_i, s_i \rangle$ with $s_i \in S \setminus S_i, 1 \leq i \leq n$, or equivalently, by
  \item all Steinberg symbols $\{1+x_i, 1+s_i\}$ with $s_i \in S \setminus S_i, 1 \leq i \leq n.$
\end{enumerate}
The number of Dennis-Stein symbols listed here is $n(p^n-p)$. For a fixed number $t$, $2 \leq t \leq n$, write $\widetilde{x}_{l_j} = (\prod_{j=1}^{t} x_{l_j}^{\lambda_{l_j}})/x_{l_j}$, where $1 \leq j \leq t$, $\widetilde{x}_{l_j} \neq 1$, $1 \leq \lambda_{l_j} \leq p-1$, $1 \leq l_j \leq n$, and the $l_j$'s are arranged by size of subscript. The number of generators as $\langle x_{l_1}, \widetilde{x}_{l_1} \rangle$ listed above is $\binom{n}{t}(p-1)^{t}$.

By Lemma \ref{scholium}, the symbols $\langle x_{l_j}, \widetilde{x}_{l_j} \rangle$ ($1 \leq j \leq t$) are linearly dependent:
$$\prod_{j=1}^{t}\langle x_{l_j}, \widetilde{x}_{l_j} \rangle^{\lambda_j}=1,$$
thus $K_2(\mathbb{F}_p[G])$ has $p$-rank at most
$$n(p^n-p) - \sum_{i=2}^{n}\binom{n}{i}(p-1)^{i} = (n-1)(p^n-1),$$
which coincides with the previous estimate. It follows that the set $T \setminus (T_1 \cup T_2)$ forms a basis for $K_2(\mathbb{F}_p[G])$.
\end{proof}

\begin{lemma} \label{Cp-fir}
Let $\sigma$ be a generator of $C_p$, where $k \geq 2$. For $p=2$, $K_2(\mathbb{Z}[C_p]/(p^k))$ is an elementary abelian 2-group of rank 2 generated by $\langle \sigma-1,\sigma-1\rangle$ and $\langle -2,-2\rangle$. For an odd prime $p$, $K_2(\mathbb{Z}[C_p]/(p^k))$ is a cyclic group of order $p$ generated by $\langle \sigma-1,(\sigma-1)^{p-1}\rangle$.
\end{lemma}

\begin{proof}
For $p=2$, $SK_1(\mathbb{Z}[C_2],(2^k))$ is trivial as stated in \cite[Theorem 1.10(c)]{alperin1987sk}. Moreover, $K_2(\mathbb{Z}[C_2]/(2^k)) \cong K_2(\mathbb{Z}[C_2]/(4))$ is an elementary abelian 2-group of rank 2, which is in accordance with \cite[Theorem 2.1]{smallK2}. Then by considering the algebraic $K$-theory exact sequence associated to the pair $(\mathbb{Z}[C_2],2^k\mathbb{Z}[C_2])$, $K_2(\mathbb{Z}[C_2])$ maps isomorphically onto $K_2(\mathbb{Z}[C_2]/(2^k))$. It is well known that $K_2(\mathbb{Z}[C_2])$ is generated by $\{\sigma,\sigma\}$ and $\{-1,-1\}$ as mentioned in \cite{stein1980excision}. Note that $(\sigma-1)^2 = 2(\sigma-1)$, and since $K_2(\mathbb{Z}[C_2]/(4))$ has exponent 2, then by Proposition 1.1 in \cite{alperin1987sk}, $K_2(\mathbb{Z}[C_2]/(4))$ is generated by $\langle \sigma-1,\sigma-1\rangle$ and $\langle -2,-2\rangle$.

For $p$ an odd prime, by \eqref{K2c-order}, the inverse limit of $K_2(\mathbb{Z}[C_p]/(p^k))$ is a cyclic group of order $p$. To prove our result, it suffices to show that $K_2(\mathbb{Z}[C_p]/(p^k))$ is a cyclic $p$-group of order $p$ generated by $\langle \sigma-1,(\sigma-1)^{p-1}\rangle$. Let $P$ be the maximal ideal of $A=\mathbb{Z}[\zeta_p]$ generated by $\pi = \zeta_p-1$. Then by algebraic number theory, $pA={\pi}^{p-1}A$. As a corollary of \cite[Theorem 4.3]{dennis1975k}, for any $k\geq 2$, $K_2(A/(p^k))$ is isomorphic to $K_2(A/P^{p})$, which is a cyclic group of order $p$. Moreover, by Corollary 3.3 and the exact sequence (3.1) in \cite[p.~209]{dennis1975k}, $K_2(A/P^{p})$ is isomorphic to $K_2(A/P^{p}, P^{p-1}/P^{p})$, which is generated by $\{1+\pi,1+{\pi}^{p-1}\}$ (see \cite[Lemma 3.2]{dennis1975k}). Note that ${\pi}^p=0$ holds in $A/P^{p}$. Then,
$$\langle \pi,{\pi}^{p-1}\rangle = \langle \pi,{\pi}^{p-1}\rangle \langle 1,{\pi}^{p-1}\rangle = \langle 1+\pi,{\pi}^{p-1}\rangle = \langle 1+\pi,-{\pi}^{p-1}\rangle^{-1} = \{1+\pi,1+{\pi}^{p-1}\}^{-1}, $$
thus $K_2(A/P^{p}, P^{p-1}/P^{p})$ is also generated by $\langle \pi,{\pi}^{p-1}\rangle$. Consequently, for any $k\geq 2$, $K_2(A/(p^k))$ is generated by $\langle \pi,{\pi}^{p-1}\rangle$. According to \cite[Theorem 2.10]{alperin1985sk}, the homomorphism $$\chi: K_2(\mathbb{Z}[C_p]/(p^k)) \rightarrow K_2(A/(p^k))$$ induced by the canonical map of rings is surjective. Obviously $\langle \sigma-1,(\sigma-1)^{p-1}\rangle$ maps  onto $\langle \pi,{\pi}^{p-1}\rangle$, and hence our result follows.
\end{proof}

\begin{remark}
There exists another approach to prove the triviality of $\{1+\pi, 1+{\pi}^{p-1}\}$. For more details, see \cite[pp.~269-270]{roberts2006k2}. From the proof of Lemma \ref{Cp-fir}, we conclude that $\langle \pi,{\pi}^{p-1}\rangle = \{1+{\pi}^{p-1}, 1+\pi\}$. Moreover, in $K_2(\mathbb{Z}[C_p]/(p^k))$ we have $\langle \sigma-1,(\sigma-1)^{p-1}\rangle = \{1+(\sigma-1)^{p-1}, \sigma\}$.
\end{remark}

\begin{theorem} \label{nontrivial2}
For any $k \geq 2$, $K_2(\mathbb{Z}[G]/(p^k))$ is an elementary abelian $p$-group of rank $(n-1)p^n+1+\tau(p)$. The set $T \setminus T_1$ together with $\langle -p, -p \rangle$ (if $p=2$) forms a basis for $K_2(\mathbb{Z}[G]/(p^k))$.
\end{theorem}

\begin{proof}
By \eqref{K2c-order}, the inverse limit of $K_2(\mathbb{Z}[G]/(p^k))$ has $p$-rank $(n-1)p^n+1+\tau(p)$. Thus, the $p$-rank of $K_2(\mathbb{Z}[G]/(p^k))$ is at most $(n-1)p^n+1+\tau(p)$.

On the other hand, let $G_i = \langle \sigma_i \rangle$ for $1 \leq i \leq n$. Then $\mathbb{F}_p[G_i] \cong \mathbb{F}_p[x_i]/({x_i}^p)$. By Corollary 4.4(a) in \cite{dennis1975k}, $K_2(\mathbb{F}_p[x_i]/({x_i}^p))$ is trivial. According to Lemma \ref{K2-surjection}, for any $k \geq 2$, $K_2(\mathbb{Z}[G]/(p^k))$ maps onto the elementary abelian $p$-group $K_2(\mathbb{F}_p[G])$, which has rank $(n-1)(p^n-1)$ (see Lemma \ref{generators}). 

Since $K_2(\mathbb{F}_p[G_i])$ is trivial, the kernel contains $n+\tau(p)$ additional elements: $\langle -p, -p \rangle$ (if $p=2$) and $\langle x_i, x_i^{p-1} \rangle$ for $1 \leq i \leq n$, which are clearly linearly independent. Following the proof of Lemma \ref{Cp-fir}, $\langle x_i, x_i^{p-1} \rangle$ maps  onto the generator of $K_2(\mathbb{Z}[G_i]/(p^k))$. It follows that the $p$-rank of $K_2(\mathbb{Z}[G]/(p^k))$ is at least $(n-1)p^n+1+\tau(p)$, which completes the proof.
\end{proof}

\section{The explicit structure of $K_2(\mathbb{Z}[G]/p^k\Gamma)$ for $k\geq n$}
According to \cite[Lemma 3.3]{ChenHong2014}, $\Gamma =\sum_{H\leq G} \mathbb{Z}[G]e_{H},$ where $e_{H}= \sum_{h\in H}h/|H|$ is an idempotent element. Let $J=|G|\Gamma$, $\widetilde{G}=\sum_{g\in G}g$, and let $I$ be the ideal of $\mathbb{Z}[G]$ generated by $\{|G|e_{H}|H\leq G, H\neq G\}$. Then $J=I+(\widetilde{G})\mathbb{Z}$. Note that $I\subseteq p\mathbb{Z}[G]$, which implies $p(\widetilde{G})\mathbb{Z} \subseteq J\cap p\mathbb{Z}[G]=I$. It follows that $pJ\subseteq I \subseteq J$, and there is a Cartesian square
\begin{equation}\label{Cartesian}
  \begin{CD}
\mathbb{Z}[G]/(p\mathbb{Z}[G] \cap J) @>>>  \mathbb{Z}[G]/J  \\
@VVV @VVV  \\
\mathbb{Z}[G]/p\mathbb{Z}[G]  @>>>  \mathbb{Z}[G]/(p\mathbb{Z}[G]+J)
\end{CD}
\quad
\mbox{i.e., }
\quad
\begin{CD}
\mathbb{Z}[G]/I @>>>  \mathbb{Z}[G]/J  \\
@VVV @VVV  \\
\mathbb{F}_p[G] @>>>  \mathbb{F}_p[G]/(\widetilde{G})
\end{CD}
\end{equation}
In \cite{ChenHong2014}, the square \eqref{Cartesian} was employed to provide a lower bound for the order of $K_2(\mathbb{Z}[G]/|G|\Gamma)$. To prove Theorem \ref{main-thm}, several technical lemmas are needed. By slightly modifying the proof of Lemma \ref{Cp-fir}, we obtain the following:

\begin{lemma} \label{Cp-sec}
Let $\sigma$ be a generator of $C_p$, and let $\Gamma_1$ be the maximal $\mathbb{Z}$-order of $\mathbb{Z}[C_p]$. For $k\geq 2$, if $p=2$, $K_2(\mathbb{Z}[C_p]/p^k\Gamma_1)$ is an elementary abelian 2-group of rank 2 generated by $\langle \sigma-1,\sigma-1\rangle$ and $\langle -2,-2\rangle$. If $p$ is an odd prime, $K_2(\mathbb{Z}[C_p]/p^k\Gamma_1)$ is a cyclic group of order $p$ generated by $\langle \sigma-1,(\sigma-1)^{p-1}\rangle$.
\end{lemma}

\begin{lemma} \label{result1}
Let $n\geq 2$. Then $K_2(\mathbb{F}_p[G]/(\widetilde{G}))$ is an elementary abelian $p$-group of rank $(n-1)(p^n-2)-1$, and the set $T\backslash (T_1\cup T_2\cup T_3)$ forms a basis for $K_2(\mathbb{F}_p[G]/(\widetilde{G}))$.
\end{lemma}

\begin{proof}
By \cite[Lemma 4.1]{ChenHong2014}, $K_2(\mathbb{F}_p[G], (\widetilde{G}))$ is an elementary abelian $p$-group of rank $n$ with generators $\langle \sigma_i, \widetilde{G}\rangle$ for $1\leq i\leq n$. Note that $(\sigma_i-1)\widetilde{G}=0$, and
$$\widetilde{G}=\prod_{j=1}^{n}\left(\sum_{i=0}^{p-1}(\sigma_j)^i\right)\equiv \prod_{j=1}^{n}(\sigma_j-1)^{p-1} \pmod{p\mathbb{Z}[G]}. $$
It follows that
$$\langle \sigma_i, \widetilde{G}\rangle = \langle \sigma_i-1, \widetilde{G}\rangle \langle 1, \widetilde{G}\rangle =\langle \sigma_i-1, \widetilde{G}\rangle = \langle \sigma_i-1, \prod_{j=1}^{n}(\sigma_j-1)^{p-1} \rangle =  \langle x_i, \prod_{j=1}^{n} x_j^{p-1} \rangle.$$ 

As the set $\{\langle x_i, \prod_{j=1}^{n} x_j^{p-1} \rangle \mid 1 \leq i \leq n\}$ is contained in a basis for $K_2(\mathbb{F}_p[G])$ \cite{zhang2019some} and each element has order $p$, it follows that $K_2(\mathbb{F}_p[G], (\widetilde{G}))$ is a direct summand of $K_2(\mathbb{F}_p[G])$. Consequently, the following short exact sequence splits:
\begin{equation}\label{splits1}
1 \rightarrow K_2(\mathbb{F}_p[G], (\widetilde{G})) \rightarrow K_2(\mathbb{F}_p[G]) \rightarrow K_2(\mathbb{F}_p[G]/(\widetilde{G})) \rightarrow 1.
\end{equation}
The result then follows from Lemma \ref{generators}.
\end{proof}

\begin{lemma} \label{result2}
Let $n\geq 2$. Then $K_2(\mathbb{Z}[G]/I)$ is an elementary abelian $p$-group of rank $(n-1)p^n+1+\tau(p)$, and the set $T\backslash T_1$ together with $\langle -p, -p \rangle$ (if $p$=2) forms a basis for $K_2(\mathbb{Z}[G]/I)$.
\end{lemma}

\begin{proof}
As described in \cite[p.~1587]{ChenHong2014}, the canonical ring homomorphism $f: \mathbb{Q}[G]\rightarrow \mathbb{Q}[G_i]$ maps $|G|\Gamma$ to $|G|\Gamma_{i}$, inducing a surjective map of rings $\mathbb{Z}[G]/|G|\Gamma \rightarrow \mathbb{Z}[G_i]/|G|\Gamma_{i}$. Thus, the map $g: K_2(\mathbb{Z}[G]/I)\rightarrow K_2(\mathbb{F}_p[G])$ is surjective by Lemma \ref{K2-surjection}. Clearly, $\ker(g)$ contains $\langle \sigma_i-1, (\sigma_i-1)^{p-1}\rangle$ for $1\leq i \leq n$ and $\langle -p, -p \rangle$ (if $p=2$). By Lemma \ref{Cp-sec}, these generators have exponent $p$. Since $p^{n+1}\mathbb{Z}[G]\subseteq pJ\subseteq I\subseteq p\mathbb{Z}[G]$, the proof of Theorem \ref{nontrivial2} implies that
$$\ker(K_2(\mathbb{Z}[G]/(p^{n+1}))\rightarrow K_2(\mathbb{F}_p[G])) \cong \ker(K_2(\mathbb{Z}[G]/I)\rightarrow K_2(\mathbb{F}_p[G])).$$
Hence, $K_2(\mathbb{Z}[G]/I)$ is isomorphic to $K_2(\mathbb{Z}[G]/(p^{n+1}))$, and the result follows.
\end{proof}

\begin{theorem}\label{main-thm}
Let $G$ be an elementary abelian $p$-group of rank $n$. Then
\begin{enumerate}[(i)]
\item For $n\geq 2$, $K_2(\mathbb{Z}[G]/p^n\Gamma)$ is an elementary abelian $p$-group of rank $(n-1)(p^n-1)+\tau(p)$. The set $T\backslash (T_1\cup T_3)$ together with $\langle -p, -p \rangle$ (if $p$=2) forms a basis for $K_2(\mathbb{Z}[G]/p^n\Gamma)$.
\item For any $k\geq n+1$, $K_2(\mathbb{Z}[G]/p^k\Gamma)$ is an elementary abelian $p$-group of rank $(n-1)p^n+1+\tau(p)$. The set $T\backslash T_1$ together with $\langle -p, -p \rangle$ (if $p$=2) forms a basis for $K_2(\mathbb{Z}[G]/p^k\Gamma)$.
\end{enumerate}
\end{theorem}

\begin{proof}
\begin{enumerate}[(i)]
	\item For $n \geq 2$, the Cartesian square \eqref{Cartesian} induces a natural commutative diagram with exact rows:
\begin{equation*}
\begin{CD}
@.    K_2(\mathbb{Z}[G]/I, J/I) @>{f_1}>>  K_2(\mathbb{Z}[G]/I)@>{f_2}>>  K_2(\mathbb{Z}[G]/J) @>>> 1\\
@.  @V{g_1}VV  @V{g_2}VV @V{g_3}VV @. \\
1@>>>  K_2(\mathbb{F}_p[G], (\widetilde{G})) @>>>  K_2(\mathbb{F}_p[G]) @>>>  K_2(\mathbb{F}_p[G]/(\widetilde{G})) @>>> 1
\end{CD}
\end{equation*}
where the bottom row is given by \eqref{splits1}. By Lemma \ref{K2-surjection}, $f_2$, $g_2$, and $g_3$ are surjective. Lemma 4.2 in \cite{ChenHong2014} shows that $g_1$ is an isomorphism, which implies $K_2(\mathbb{Z}[G]/I, J/I)$ is an elementary abelian $p$-group and $f_1$ is injective.
A corollary of Lemma \ref{K2-surjection} and Theorem \ref{nontrivial2} shows that both $K_2(\mathbb{Z}[G]/I)$ and $K_2(\mathbb{Z}[G]/J)$ are elementary abelian $p$-groups. Consequently, the top exact sequence splits. By the Snake Lemma, $\ker(g_2)\cong \ker(g_3)$. Since $g_2$ and $g_3$ are split surjections, we have
\begin{equation*}
  K_2(\mathbb{Z}[G]/J)\cong K_2(\mathbb{F}_p[G]/(\widetilde{G})) \oplus \ker\big(K_2(\mathbb{Z}[G]/I)\rightarrow K_2(\mathbb{F}_p[G])\big).
\end{equation*}
Based on the $p$-ranks of $K_2(\mathbb{F}_p[G])$, $K_2(\mathbb{F}_p[G]/(\widetilde{G}))$, and $K_2(\mathbb{Z}[G]/I)$ provided in Lemmas \ref{generators}, \ref{result1}, and \ref{result2} respectively, $K_2(\mathbb{Z}[G]/p^n\Gamma)$ is an elementary abelian $p$-group of rank $(n-1)(p^n-1)+\tau(p)$. 
    
    \item For the case $n=1$ and $k\geq 2$, the result is given by Lemma \ref{Cp-sec}. 
For $n \geq 2$ and any $k\geq n+1$, the relation $pJ\subseteq I$ implies the existence of a surjective map from $K_2(\mathbb{Z}[G]/p^k\Gamma)$ onto $K_2(\mathbb{Z}[G]/I)$. Moreover, since $K_2(\mathbb{Z}[G]/(p^k))$ maps surjectively onto $K_2(\mathbb{Z}[G]/p^k\Gamma)$, it suffices to show that $K_2(\mathbb{Z}[G]/(p^k))$ and $K_2(\mathbb{Z}[G]/I)$ have the same $p$-rank. This follows directly from Theorem \ref{nontrivial2} and Lemma \ref{result2}.
\end{enumerate} 
\end{proof}

\begin{remark}
From the proof, it can be concluded that for any $k\geq n+1$, $K_2(\mathbb{Z}[G]/(p^k)) \cong K_2(\mathbb{Z}[G]/p^k\Gamma)$.
\end{remark}

\section{Applications}

This section provides a definitive description of $SK_1(\mathbb{Z}[G], p^k\mathbb{Z}[G])$, where $p$ is an odd prime and $G$ is an elementary abelian $p$-group of rank $n$ (For $p$=2, the result is trivial, see \cite[Theorem 1.10]{alperin1987sk}). Subsequently, we address the calculation of $SK_1(\mathbb{Z}[G], p\mathbb{Z}[G])$, a problem that has remained open since it was mentioned in \cite{dennis1976lower}. To resolve these computational issues, we first discuss several necessary preliminaries.

Let $G$ be a finite abelian $p$-group, and let $\Gamma$ be the maximal $\mathbb{Z}$-order in $\mathbb{Q}[G]$. Consider an irreducible character $\chi: G\rightarrow \mathbb{C}^{*}$ of $G$. The image $\chi(G)$ is a finite cyclic group of roots of unity whose order divides $|G|$, and $\chi$ induces a surjective homomorphism from $\mathbb{Q}[G]$ to $\mathbb{Q}[\chi]$, where $\mathbb{Q}[\chi]$ denotes the cyclotomic extension of $\mathbb{Q}$ generated by $\chi(G)$. Recall that a cluster of characters $S$ for $G$ is a set of irreducible characters containing exactly one character dividing each irreducible $\mathbb{Q}[G]$-module (see \cite[Definition 2.1]{alperin1985sk}). There is a natural isomorphism $\mathbb{Q}[G]\cong \prod_{\chi \in S}\mathbb{Q}[\chi]$. Furthermore, by \cite[Proposition 2.2]{alperin1985sk}, we have $\Gamma\cong \prod_{\chi \in S}\mathbb{Z}[\chi]$ and $|G|\Gamma \subseteq \mathbb{Z}[G] \subseteq \Gamma$. An imaginary cluster of characters for $G$ is defined as $S_0= \{\chi \in S \mid \mathbb{Z}[\chi]\neq \mathbb{Z} \}$. Clearly,
$$S_0 =\left\{
\begin{array}{ccl}
 \{\chi \in S\mid |\mbox{Im}(\chi)|>2\}    &      & {\mbox{if $p=2$},} \\
 S\backslash \{\mbox{trivial character}\}  &      &  {\mbox{if $p\neq2$}.}\\
\end{array} \right.$$

\begin{lemma}\cite[Lemma 13.37]{magurn2002algebraic} \label{sk1-sur}
If $R$ is a commutative ring with ideals $J' \subseteq J$ of finite index, then the inclusion on $SL$ induces a surjective homomorphism $SK_1(R,J' ) \rightarrow SK_1(R,J)$.
\end{lemma}

\begin{proposition} \label{sk1-inverses}
Let $G$ be a finite abelian $p$-group and $S_0$ be an imaginary cluster of characters for $G$. Then
$$\varprojlim_{k} SK_1(\mathbb{Z}[G],(p^k)) \cong \prod_{\chi\in S_0} \mbox{Im}(\chi). $$
\end{proposition}
\begin{proof}
 Set $|G|=p^{k_0}$. Then for any $k\geq k_0$, the inclusion $p^{k}\Gamma \subseteq p^{k-k_0} \mathbb{Z}[G] \subseteq p^{k-k_0} \Gamma$ holds. Therefore,
\begin{align}
\varprojlim_{k} SK_1(\mathbb{Z}[G],(p^k)) 
&= \varprojlim_{k\geq k_0} SK_1(\mathbb{Z}[G],(p^k)) \nonumber \\
&\cong \varprojlim_{k\geq k_0} SK_1(\mathbb{Z}[G],p^k\Gamma) &&\mbox{ (by Lemma \ref{sk1-sur})} \nonumber \\
   &\cong \varprojlim_{k\geq k_0} SK_1(\Gamma,p^k\Gamma)
    &&\mbox{(by \cite[Theorem 1.3(f)]{alperin1985sk})} \nonumber \\
    & \cong \varprojlim_{k\geq k_0} \prod_{\chi \in S} SK_1(\mathbb{Z}[\chi],p^k\mathbb{Z}[\chi])     &&\mbox{ ($SK_1$ commutes with products)}   \nonumber \\
    & \cong \prod_{\chi\in S_0} \mbox{Im}(\chi). \nonumber     &&\mbox{ (by Corollary 4.3 in \cite{bass1967solution})}
\end{align}
\end{proof}

\begin{theorem}\label{relative-sk1}
Let $p$ be an odd prime, and let $G$ be an elementary abelian $p$-group of rank $n$. Then:
\begin{enumerate}[(i)]
	\item For any $k\geq 2$, $SK_1(\mathbb{Z}[G], p^k\mathbb{Z}[G])$ is an elementary abelian $p$-group of rank $\dfrac{p^n-1}{p-1}$.
	\item $SK_1(\mathbb{Z}[G], p\mathbb{Z}[G])$ is an elementary abelian $p$-group of rank $\dfrac{p^n-1}{p-1}-n$.
\end{enumerate}
\end{theorem}
\begin{proof}
\begin{enumerate}[(i)]
	\item Since the $SK_1$ of finite rings is trivial \cite[p.~469, Proposition 3.10]{bass1968algebraic}, we have the following natural commutative diagram with exact rows:
\begin{equation*}
 \begin{CD}
K_2(\mathbb{Z}[G]) @>>> K_2(\mathbb{Z}[G]/(p^{k+1}))@>>> SK_1(\mathbb{Z}[G], (p^{k+1})) @>>> SK_1(\mathbb{Z}[G])   @>>> 1 \\
@| @VV{\cong}V @VVV  @|   @|  \\
K_2(\mathbb{Z}[G]) @>>> K_2(\mathbb{Z}[G]/(p^{k}))@>>> SK_1(\mathbb{Z}[G], (p^{k})) @>>> SK_1(\mathbb{Z}[G])   @>>> 1 \\
\end{CD}
\end{equation*}
By Theorem \ref{nontrivial2}, $K_2(\mathbb{Z}[G]/(p^{k}))$ is stable for $k\geq 2$. It then follows from the Snake Lemma that $SK_1(\mathbb{Z}[G], (p^{k}))$ is also stable for $k\geq 2$, and thus coincides with its inverse limit. Hence, by Proposition \ref{sk1-inverses}, $SK_1(\mathbb{Z}[G], (p^k))$ is an elementary abelian $p$-group of rank $\dfrac{p^n-1}{p-1}$. 

	\item Let $G=\prod_{i=1}^{n}G_i$, where $G_i$ denotes a cyclic group of order $p$, and let $\Gamma_i$ be the maximal $\mathbb{Z}$-order of $\mathbb{Z}[G_i]$. Following the proof of Proposition \ref{sk1-inverses}, we have
$$SK_1(\mathbb{Z}[G_i],p^k\Gamma_i) \cong SK_1(\Gamma_i, p^k\Gamma_i)\cong SK_1(\mathbb{Z}, p^k\mathbb{Z}) \oplus SK_1(\mathbb{Z}[\zeta_p], p^k\mathbb{Z}[\zeta_p]) = C_p.$$
Consider the natural commutative diagram:
\begin{equation}\label{connect}
 \begin{CD}
K_2(\mathbb{Z}[G]/(p^k)) @>{\partial}>> SK_1(\mathbb{Z}[G], (p^k))  \\
@V{p_{i*}}VV @V{p_{i*}}VV  \\
K_2(\mathbb{Z}[G_i]/(p^k)) @>{\partial}>>  SK_1(\mathbb{Z}[G_i], (p^k))
\end{CD}
\end{equation}
where $\partial$ denotes the connecting homomorphism and $p_{i*}$ is the map induced by the projection $p_{i}: G \rightarrow G_i$.
Let $L=\ker(K_2(\mathbb{Z}[G]/(p^k))\rightarrow K_2(\mathbb{F}_p[G]))$. From the proof of Theorem \ref{nontrivial2}, $L$ is an elementary abelian $p$-group of rank $n$, and $L\cong \prod_{i=1}^{n} K_2(\mathbb{Z}[G_i]/(p^k))$.
Since $SK_1(\mathbb{Z}[G_i])$ is trivial, the algebraic $K$-theory sequence for $(\mathbb{Z}[G_i], (p^k))$ shows that $K_2(\mathbb{Z}[G_i]/(p^k))$ maps onto $SK_1(\mathbb{Z}[G_i],(p^k))$. By Lemma \ref{sk1-sur}, $SK_1(\mathbb{Z}[G_i],(p^k))$ maps onto $SK_1(\mathbb{Z}[G_i], p^k\Gamma_i)$, which is a cyclic group of order $p$. Moreover, by Lemma \ref{Cp-fir}, $K_2(\mathbb{Z}[G_i]/(p^k))$ is cyclic of order $p$. Thus, $K_2(\mathbb{Z}[G_i]/(p^k)) \cong SK_1(\mathbb{Z}[G_i],(p^k)) \cong C_p$.
As \eqref{connect} commutes and $SK_1(\mathbb{Z}[G], (p^k))$ has exponent $p$, $SK_1(\mathbb{Z}[G_i],(p^k))$ is isomorphic to a direct summand of $SK_1(\mathbb{Z}[G], (p^k))$. Let $L_1$ be the image of $L$ in $SK_1(\mathbb{Z}[G], (p^k))$. Then $L_1$ is an elementary abelian $p$-group of rank $n$, and $L_1 \cong \prod_{i=1}^{n} SK_1(\mathbb{Z}[G_i], (p^k))$.
For any $k\geq 2$, we obtain another commutative diagram with exact rows:
\begin{equation*}
 \begin{CD}
K_2(\mathbb{Z}[G]) @>>> K_2(\mathbb{Z}[G]/(p^k))/L@>>> SK_1(\mathbb{Z}[G], (p^k))/L_1 @>>> SK_1(\mathbb{Z}[G])   @>>> 1 \\
@| @VV{\cong}V @VVV  @|   @|  \\
K_2(\mathbb{Z}[G]) @>>> K_2(\mathbb{Z}[G]/(p))@>>> SK_1(\mathbb{Z}[G], (p)) @>>> SK_1(\mathbb{Z}[G])   @>>> 1 \\
\end{CD}
\end{equation*}
The second vertical arrow is an isomorphism by definition. By the Snake Lemma, the third vertical arrow is also an isomorphism. Hence, $SK_1(\mathbb{Z}[G], (p))$ is an elementary abelian $p$-group of rank $\dfrac{p^n-1}{p-1}-n$.
\end{enumerate}
\end{proof}

\begin{remark}
Let $p$ be an odd prime. The proof of Theorem \ref{relative-sk1} implies that for $k\geq 2$, $SK_1(\mathbb{Z}[G], (p^k)) \cong SK_1(\Gamma, (p^k))$. Furthermore, according to Theorems A.2 and A.3 in \cite{alperin1985sk}, $SK_1(\Gamma, (p^k)) \cong K_2(\Gamma/(p^k))$. Thus, by \cite[Proposition 2.5]{alperin1987sk}, the image of $K_2(\mathbb{Z}[G]/(p^k))$ in $SK_1(\mathbb{Z}[G], (p^k))$ is an elementary abelian $p$-group of rank $\dbinom{p+n-1}{p}$. Alternatively, for $k=1$, the image has $p$-rank $\dbinom{p+n-1}{p}-n$. Consequently, one can recover the classical result that $SK_1(\mathbb{Z}[G])$ is an elementary abelian $p$-group of rank $\dfrac{p^n-1}{p-1}-\dbinom{p+n-1}{p}$. 
\end{remark}

\section*{Acknowledgement}
The author sincerely appreciate the constructive feedback and insightful comments provided by the editors and anonymous reviewers.

\section*{Disclosure statement}
The author declares that there are no competing interests of a financial or personal nature.

\bibliography{mybibfile}

\end{document}